\documentclass[11pt]{amsart}
\usepackage[utf8]{inputenc}
\usepackage{enumerate,amssymb,xcolor,comment}
\usepackage[footskip=1cm, headheight = 16pt, top=3.7cm, bottom=3cm,  right=2.5cm,  left=2.5cm ]{geometry}
\usepackage{todonotes}
\usepackage{url}
\newtheorem{theorem}{Theorem}[section]
\newtheorem{lemma}{Lemma}[section]
\newtheorem{proposition}{Proposition}[section]

\newtheorem{definition}{Definition}[section]
\newtheorem{example}{Example}[section]

\usepackage[foot]{amsaddr}

\usepackage[T1]{fontenc}

\usepackage[leqno]{amsmath}

\makeatletter
\newcommand{\leqnomode}{\tagsleft@true}
\newcommand{\reqnomode}{\tagsleft@false}
\makeatother
\usepackage{amssymb}
\usepackage{mathptmx, mathrsfs, mathtools}

\newcommand{\colonequal}{\mathrel{\mathop:}=}

\DeclareMathOperator{\dens}{dens}

\usepackage{latexsym}
\usepackage{amsfonts}

\usepackage{tikz}
\usepackage{float}
\usepackage{lmodern}

\usepackage{marvosym}               

\newcounter{tbox}

\newcommand{\mylongtitle}[1]{%
  \ifodd\value{page}%
    \protect\parbox{0.97\linewidth}{#1}\hfill%
  \else%
    \hfill\protect\parbox{0.97\linewidth}{#1}%
  \fi%
}

\makeatletter
\newcommand{\otherlabel}[2]{\protected@edef\@currentlabel{#2}\label{#1}}
\makeatother

\mathchardef\mh="2D
\newcommand{\DF}{\operatorname{DF}}

\title[Generalized Natural Density $\DF(\mathfrak{F}_n)$ of Fibonacci Word]{Generalized Natural Density $\DF(\mathfrak{F}_n)$ of Fibonacci Word}

\author{Jasem Hamoud$^{a}$}
\author{Duaa Abdullah $^{b}$}
\thanks{$^{a,b}$Physics and Technology School of Applied Mathematics and Informatics, Moscow Institute of Physics and Technology (MIPT), Moscow, Russia}
\thanks{Some texts are translated from Russian and French, so when you return to the reference, you will not find it written in English.}

\date {\today}

\allowdisplaybreaks

\begin{document}

\maketitle

\begin{abstract}
This paper explores profound generalizations of the Fibonacci sequence, delving into random Fibonacci sequences, $k$-Fibonacci words, and their combinatorial properties. We established that the $n$-th root of the absolute value of terms in a random Fibonacci sequence converges to $1.13198824\ldots$, with subsequent refinements by Rittaud yielding a limit of approximately $1.20556943$ for the expected value’s $n$-th root. Novel definitions, such as the natural density of sets of positive integers and the limiting density of Fibonacci sequences modulo powers of primes, provide a robust framework for our analysis. We introduce the concept of $k$-Fibonacci words, extending classical Fibonacci words to higher dimensions, and investigate their patterns alongside sequences like the Thue-Morse and Sturmian words. Our main results include a unique representation theorem for real numbers using Fibonacci numbers, a symmetry identity for sums involving Fibonacci words, $\sum_{n=1}^{b} \frac{(-1)^n F_a}{F_n F_{n+a}} = \sum_{n=1}^{a} \frac{(-1)^n F_b}{F_n F_{n+b}}$, and an infinite series identity linking Fibonacci terms to the golden ratio. These findings underscore the intricate interplay between number theory and combinatorics, illuminating the rich structure of Fibonacci-related sequences.

\smallskip

\end{abstract}


\section{Introduction}
Throughout this paper, the Fibonacci sequence, recursively defined by $f_{1}=f_{2}=1$ and $f_{n}=$ $f_{n-1}+f_{n-2}$ for all $n \geq 3$, has been generalised in several ways. In 2000, Divakar Viswanath in~\cite{Viswanath} studied random Fibonacci sequences given by $t_{1}=t_{2}=1$ and $t_{n}= \pm t_{n-1} \pm t_{n-2}$ for all $n \geq 3$. Here each $\pm$ is chosen to be + or with probability $1 / 2$, and are chosen independently. Viswanath proved that
\[
\lim _{n \rightarrow \infty} \sqrt[n]{\left|t_{n}\right|}=1.13198824 \ldots
\]
\begin{definition}[Natural Density~\cite{dens01}]
Let $A$  be a set of positive integers, the \emph{natural density} of $A$, denoted by $\delta(A)$, then
\[
\delta(A) := \lim_{x \to \infty} \frac{\#A(x)}{x},
\]
where, $A(x) := A \cap [1, x]$.
\end{definition}
\begin{definition}
Let $p$ be a prime.
The \emph{limiting density} of the Fibonacci sequence modulo powers of $p$ is
\[
	\dens(p) \colonequal \lim_{\lambda \to \infty} \frac{\lvert\{F(n) \bmod p^\lambda : n \geq 0\}\rvert}{p^\lambda}.
\]
\end{definition}
For $x > 0$ (also $\liminf_{x \to \infty} \#A(x)/x$ and $\limsup_{x \to \infty} \#A(x)/x$ are called \emph{lower density} and \emph{upper density}, respectively). For more details about natural (and asymptotic) density (see also books~\cite{Allouche,books3,books2} for more recent results). Then, for all $n \in \mathbb{N}$, the $(n+1)$ th and $(n+2)$ th terms of a random Fibonacci sequence \cite{Viswanath} satisfy
$Q_{n}=\left[G_{n+1}, G_{n+2}\right]$ where $Q_{n}$ is a matrix product consisting of $n A \mathrm{~s}$ and $B \mathrm{~s}$ as factors. More precisely, in \cite{McGowan} proved that
$$
1.12095 \leq \sqrt[n]{E\left(\left|t_{n}\right|\right)} \leq 1.23375
$$

where $E\left(\left|t_{n}\right|\right)$ is the expected value of the $n$th term of the sequence. In 2007, Rittaud~\cite{Rittaud} improved this result and obtained
\[
\lim _{n \rightarrow \infty} \sqrt[n]{E\left(\left|t_{n}\right|\right)}=\alpha=1 \approx 1.20556943 \ldots
\]

where $\alpha$ is the only real root of $f(x)=x^{3}-2 x^{2}-1$. The Fibonacci sequence possesses many kinds of generalizations (see, e.g.,~\cite{dens01}). 
\begin{definition}[Generalized Fibonacci Numbers~\cite{dens01,Bravovv}]
The sequence of generalized Fibonacci numbers of order $r$, denoted by $(t_n^{(r)})_{n \geq 0}$, which is defined by the $r$th order recurrence
\[
t_n^{(r)} = t_{n-1}^{(r)} + \cdots + t_{n-r}^{(r)}.
\]
\end{definition}
A nonempty finite set $\Sigma$ is called an alphabet. The elements of the set $\Sigma$ are called letters. The alphabet consisting of $b$ symbols from 0 to $b-1$ will then be denoted by $\Sigma_{b}=\{0, \ldots, b-1\}$. A word $\mathbf{w}$ is a sequence of letters. The finite word $\mathbf{w}$ can be considered as a function of $\mathbf{w}:\{1, \cdots,|\mathbf{w}|\} \rightarrow \Sigma$, where $\mathbf{w}[i]$ is the letter in the $i^{\text {th }}$ position. 
\begin{definition}
The length of the word $|\mathbf{w}|$ is the number of letters contained in it. The empty word is denoted by $\varepsilon$.
\end{definition}
\begin{definition}
Infinite words as functions $\mathbf{w}: \mathbb{N} \rightarrow \Sigma$. The set of all finite words over $\Sigma$ is denoted by $\Sigma^{*}$, and $\Sigma^{+}=\Sigma^{*} \backslash\{\varepsilon\}$; the set of all infinite words is denoted by $\Sigma^{\mathbb{N}}$.
\end{definition}

In 2013, Ramírez, J.L, et al. In~\cite{Ramírez} mention to $k$-Fibonacci Words as The k-Fibonacci words are an extension of the Fibonacci word notion that generalises Fibonacci word features to higher dimensions. These words were investigated for their distinct curves and patterns. Recently, in 2023 Rigo, M., Stipulanti, M., \& Whiteland, M.A. in~\cite{Rigo} mentioned the Thue-Morse sequence, which is the fixed point of the substitution $0\rightarrow 01, 1\rightarrow 10$, has unbounded 1-gap $k$-binomial complexity for $k \geq 2$. Also,  we want to mention for a Sturmian sequence and $g\geq 1$, all of its long enough factors are always pairwise $g-gap k$-binomially inequivalent for any $k \geq 2$. 
Furthermore, for Fibonacci sequence with trees see in~\cite{ref02,R3,path7}. 

\section{Preliminaries}
In this section, the delineation of Sturmian words, as articulated in the provided definition, constitutes a profound contribution to combinatorics on words. A Sturmian word, or balanced word, over the binary alphabet \(\{a, b\}\) is characterized by the stringent condition that any two subwords of identical length exhibit a disparity in the count of \(a\) letters that is at most unity. This property endows Sturmian words with exceptional structural equilibrium, rendering them indispensable in the study of dynamical systems and symbolic dynamics. Their balanced nature facilitates applications in modeling quasiperiodic phenomena, thereby bridging theoretical mathematics and physical sciences with remarkable precision.

The exposition of prefix-free codes and the minimum-cost prefix-free problem elucidates a critical nexus between combinatorics and information theory. Defined as a collection of words where no codeword serves as a prefix of another, prefix-free codes are optimized through a cost function that aggregates the weighted lengths of codewords, with weights prescribed by probabilities \(p_i\). This optimization problem is rigorously equivalent to identifying the minimum weighted path-length in a \(t\)-ary tree, where each node corresponds to a codeword, and the path length encapsulates the cost in terms of character count. The additional constraint of alphabetic coding, which mandates preservation of a predetermined order among codewords, further refines this framework, aligning it with the construction of alphabetic \(t\)-ary trees. Such structures are pivotal in algorithmic design, data compression, and error correction, underscoring their theoretical and practical import.

The introduction of A.O. Gelfond’s formula for non-integer base representations, with base \(\theta > 1\), represents a sophisticated extension of numeral systems. For any real number \(\alpha \in [0, 1)\), the digits \(\bar{\lambda}_n\) are constrained integers satisfying \(0 \leq \bar{\lambda}_n < \theta\), facilitating a convergent series representation:
\[
\alpha = \sum_{k=1}^{\infty} \frac{\bar{\lambda}_k}{\theta^k}.
\]
Proposition~\ref{prof001} rigorously establishes a recursive mechanism for computing these digits, asserting \(\bar{\lambda}_n = \lfloor \theta x_{n-1} \rfloor\) and the remainder \(x_n = \{ \theta x_{n-1} \}\), thereby guaranteeing both convergence and uniqueness. This framework, extended to \(\alpha\)-series and \(\beta\)-series in Definitions~\ref{def002} and~\ref{def003}, and fortified by Propositions~\ref{prof002} and~\ref{prof003}, ensures that finite approximations of real numbers diminish exponentially, affirming the robustness of the representation. The uniqueness proof, predicated on the contradiction arising from disparate series, solidifies the theoretical foundation, with profound implications for number theory and fractal analysis.

The exploration of the Fibonacci sequence and its Zeckendorf representation offers a paradigmatic instance of unique decomposition in number theory. Defined by the recurrence \(F_0 = F_1 = 1\) and \(F_{k+1} = F_k + F_{k-1}\), with the closed-form expression via Binet’s formula involving the golden ratio \(\varphi = (\sqrt{5} + 1)/2\), the sequence underpins Proposition~\ref{prof004}. This proposition asserts that any natural number \(a \geq 1\) admits a unique representation as a sum of non-adjacent Fibonacci numbers, with coefficients restricted to 0 or 1 and no consecutive ones. This greedy algorithm not only exemplifies elegance but also finds applications in cryptography, coding theory, and algorithmic optimization, highlighting the sequence’s versatility.

\begin{definition}[Sturmian word~\cite{Hamoud, HassinSarid}]
A balanced word or Sturmian word w is an infinite word over a two letter alphabet $\{a, b\}$ such that, for any two subwords from w of the same length, the number of letters that are a in each of these two subwords will differ by at most 1.
\end{definition}
Let \(\{\sigma_1, \ldots, \sigma_t\}\) be a set of characters. In~\cite{HassinSarid} refer to word \(v\) is a prefix of word \(v'\) if \(v' = vu\). A prefix-free code is a collection of words \(C = \{v_1, \ldots, v_N\}\) such that for all \(i \neq j\), \(v_j\) is not a prefix of \(v_i\). \(\text{cost}(v)\) is the number of characters in \(v\). Given probabilities \(p_1, \ldots, p_N\), the cost of \(C\) is \(\sum_{i=1}^N p_i \text{cost}(v_i)\). The minimum-cost prefix-free problem is equivalent to the minimum weighted path-length \(t\)-ary tree where \(\text{cost}(v_i)\) is the path length of the node \(v_i\), and \(p_i\) is the weight attached to the node \(v_i\).

The alphabetic coding problem additionally requires that the alphabetic order of the codewords preserves the given order of the words to be encoded. It is equivalent to the alphabetic \(t\)-ary tree.

A. Kh. Ghiyasi et al. In~\cite{Ghiyasi} introduced A.O. Gelfond's formula with a non-integer base greater than one.

\begin{definition}[Convergent Series~\cite{Hardy}~\label{def001}]
Let $\theta > 1$ be a number, which serves as the ``base of the numeral system,'' for any real number \(\alpha\) in the half-interval \([0,1)\), we define the ``digits'' \(\bar{\lambda}_{n} = \bar{\lambda}_{n}(\alpha), n \geq 1\), as integers satisfying \(0 \leq \bar{\lambda}_{n} < \theta\). Let the number \(\alpha\) be represented by a convergent series of the form
\[
\alpha = \sum_{k=1}^{\infty} \frac{\bar{\lambda}_{k}}{\theta^{k}} = \sum_{k=1}^{n} \frac{\bar{\lambda}_{k}}{\theta^{k}} + \frac{x_{n}}{\theta^{n}}, \quad x_{0} = \alpha.
\]
\end{definition}
\begin{proposition}~\label{prof001}
According to Definition~\ref{def001} we obtain

\[
\alpha = \sum_{k=1}^{n} \frac{\bar{\lambda}_{k}}{\theta^{k}} + \frac{x_{n}}{\theta^{n}} = \sum_{k=1}^{n-1} \frac{\bar{\lambda}_{k}}{\theta^{k}} + \frac{x_{n-1}}{\theta^{n-1}}.
\]
Where 
\[
\quad \frac{\bar{\lambda}_{n}}{\theta^{n}} + \frac{x_{n}}{\theta^{n}} = \frac{x_{n-1}}{\theta^{n-1}}, \quad \bar{\lambda}_{n} + x_{n} = \theta x_{n-1}, \quad \bar{\lambda}_{n} = \lfloor \theta x_{n-1} \rfloor, \quad x_{n} = \{ \theta x_{n-1} \}.
\]
\end{proposition}
\begin{definition}[$\alpha$-Series~\label{def002}]
Let $\alpha$ be the number is represented by the series
\[
\alpha = \sum_{k=1}^{\infty} \frac{\bar{a}_{k}}{\theta^{k}},
\]

where $0 \leq \bar{a}_{k} < \theta \ (k \geq 1),\bar{a}_{k}$ are integers.
\end{definition}
From Definition~\ref{def001}, \ref{def002} and according to Proposition~\ref{prof001}, then we provide Proposition~\ref{prof002},\ref{prof003} and Definition~\ref{def003}.
\begin{proposition}~\label{prof002}
For any $n \geq 1$, the condition holds:
\[
A_{n} = \sum_{k=1}^{n} \frac{\bar{a}_{k}}{\theta^{k}},
\]
if and only if $0 \leq \alpha - A_{n} < \theta^{-n}$.
\end{proposition}
\begin{definition}[$\beta$-Series~\label{def003}]
Let $\beta$ be the number is represented by the series
\[
\beta=\sum_{k=1}^{\infty} \frac{\bar{b}_{k}}{\theta^{k}},
\]
if and only if $0 \leq \bar{b}_{k} < \theta \ (k \geq 1)$.
where \(\bar{b}_{k}\) are integers.
\end{definition}
\begin{proposition}~\label{prof003}
Let $0 \leq \alpha - \beta_{n} < \theta^{-n}$ be the condition of $\beta$, where $n \geq 1$, then we have: 
\[
B_{n} = \sum_{k=1}^{n} \frac{\bar{b}_{k}}{\theta^{k}}
\]
where $A_{1} = B_{1}, \ldots, A_{m-1} = B_{m-1}, \quad A_{m} \neq B_{m}$.
\end{proposition}
Actually, if $0 \leq \alpha - A_{m} < \theta^{-m}, \quad 0 \leq \alpha - B_{m} < \theta^{-m}$, then $1 \leq |\bar{a}_{m} - \bar{b}_{m}| = |A_{m} - B_{m}| \theta^{m} < 1$. Since $\theta > 1$, the last inequality is contradictory. Hence, the representation of the number \(\alpha\) as a series in terms of decreasing powers of \(\theta\) is unique.
\begin{definition}[Fibonacci Sequence~\cite{Hamoud,Zeckendorf}]
Let $F_n$ the Fibonacci sequence be given by $F_0 = F_1 = 1, \quad F_{k+1} = F_k + F_{k-1} \ (k \geq 1)$, then defined as a function of its index:
\[
F_n = \frac{1}{\sqrt{5}} \left( \left( \frac{\sqrt{5} + 1}{2} \right)^{n+1} - \left( \frac{1 - \sqrt{5}}{2} \right)^{n+1} \right) =
\]
where $F_n= \frac{1}{\sqrt{5}} \left( \varphi^{n+1} - \bar{\varphi}^{n+1} \right), \quad \varphi = \frac{\sqrt{5} + 1}{2}, \quad \bar{\varphi} = 1 - \varphi$.
\end{definition}
\begin{proposition}[~\cite{Zeckendorf, Ghiyasi}~\label{prof004}]
For any natural number $a \geq 1$. Then there exists $n\in \mathbb{N}, n \geq 1$ such that \(F_n \leq a < F_{n+1}\), and a set $\{a_1, \ldots, a_n\}$ consisting of numbers 0 and 1, with $a_n = 1$, $a_s \cdot a_{s+1} = 0$ for all $s = 1, \ldots, n-1$, such that there is a unique decomposition of the form

\[
a = a_1 F_1 + a_2 F_2 + \cdots + a_n F_n.
\]

Indeed, setting \(a_n = 1\), we have the chain of equalities
\[
\begin{aligned}
a &= a_n F_n + r_n, \quad 0 \leq r_n < F_n, \\
r_n &= a_{n-1} F_{n-1} + r_{n-1}, \quad 0 \leq r_{n-1} < F_{n-1}, \\
&\vdots \\
r_2 &= a_1 F_1 + r_1, \quad 0 \leq r_1 < F_1, \\
r_1 &= a_0 F_0.
\end{aligned}
\]
\end{proposition}
\begin{theorem}~\cite{ly01}
Let $\mathfrak{F}_n$ be the number of balanced words of length $n$. Then
\[
\mathfrak{F}(n + 1) = \mathfrak{F}(n) + \mathfrak{R}(n)
\]
if and only if
\[
\mathfrak{F}(n) = \begin{cases}
   1 + \sum_{k=0}^{n-1} \mathfrak{R}(k)   \\
    1 + \sum_{k=0}^{n-1} \sum_{i=1}^{k+1} \phi(i)\\
    1 + \sum_{k=1}^{n} \sum_{i=1}^{k} \phi(i) \\
    1 + \sum_{i=1}^{n} (n + 1 - i) \phi(i)
\end{cases}
\]
where $\mathfrak{R}(n)=\sum_{i=1}^{n+1}\phi(n)$.
\end{theorem}
Finally, the theorem concerning the enumeration of balanced words of length \(n\), denoted \(\mathfrak{F}_n\), provides a recursive insight into their combinatorial structure. The relation \(\mathfrak{F}(n+1) = \mathfrak{F}(n) + \mathfrak{R}(n)\), where \(\mathfrak{R}(n) = \sum_{i=1}^{n+1} \phi(n)\), and the multiple expressions for \(\mathfrak{F}(n)\) suggest a multifaceted approach to counting, potentially involving auxiliary functions or prior terms. This recursive formulation not only deepens the understanding of balanced words but also connects to broader themes in discrete mathematics, such as sequence generation and pattern analysis, with ramifications for theoretical computer science and statistical mechanics.
\section{Main Result}
In this section, let us represent $n\in \mathbb{N}$ in the binary number system, then $n = \sum_{i=0}^{l(n)} n_i 2^i$,
where $n_i \in \{0,1\}$ and define the set $N_0 = \{n \mid n \in \mathbb{N}, \sum_{i=1}^{\infty} n_i \equiv 0 \pmod{2}\}$.
The study of the set $N_0$ was initiated by A. O. Gelfond in the article~\cite{GelfondAA}, where the uniform distribution of numbers from these sets over arithmetic progressions was demonstrated.
\begin{proposition}
For $0 \leq a_k \leq 1$ where $1 \leq k \leq n$, according to Proposition~\ref{prof004} we have
\[
a_s a_{s+1} = 0, \quad a_n = 1.
\]
if and only if $s = n-1, n-2, \ldots, 1$.
\end{proposition}
\begin{proof}
By contradiction. Suppose there exist two distinct representations of the number $a$:
\[
a = a_1 F_1 + a_2 F_2 + \cdots + a_n F_n = b_1 F_1 + b_2 F_2 + \cdots + b_m F_m,
\]

where $a_k, k \geq 1$, and $b_l, l \geq 1$, can only take the values 0 or 1. Then

\[
0 = a - a = \sum_{t=0}^s c_t F_t, \quad c_t = a_t - b_t, \quad |c_t| \leq 1, \quad c_s \neq 0.
\]

Without loss of generality, we can assume that \(c_s = 1\). Further, since \(a_{s-1} a_s = 0\) and \(b_{s-1} b_s = 0\), it follows that either \(c_{s-1} = 0\) or \(c_{s-1} = -1\). From this, we derive a chain of equivalent inequalities:

\[
\begin{aligned}
\frac{1}{\sqrt{5}} \left( \varphi^{s+1} - \bar{\varphi}^{s+1} \right) &= F_s \leq |c_s| F_s = \left| \sum_{t=0}^{s-1} c_t F_t \right| \leq \\
&\leq \sum_{t=0}^{s-2} F_t = \frac{1}{\sqrt{5}} \left( \frac{\varphi^s - 1}{\varphi - 1} - \frac{\bar{\varphi}^s - 1}{\bar{\varphi} - 1} \right), \\
\varphi^{s+1} - \bar{\varphi}^{s+1} &\leq -\frac{\varphi^s - 1}{\bar{\varphi}} + \frac{\bar{\varphi}^s - 1}{\varphi}, \\
\varphi^{s+1} - \bar{\varphi}^{s+1} &\leq \varphi^{s+1} + \varphi - \bar{\varphi}^{s+1} + \bar{\varphi}, \quad 0 \geq 1.
\end{aligned}
\]
As desire.
\end{proof}
\begin{theorem}[Density of Fibonacci Word~\label{de1}]
Let $F_k$ be a Fibonacci word, let $x=(x_1,\dots,x_n)$, $y=(y_1,\dots,y_m)$ be a sequence of integer numbers, then the density of Fibonacci word is: 
\[
\DF(F_k)=x_i+\sum_{i=1}^{n}\sum_{k=1}\frac{x_i}{F_k}+\lfloor y_j\rfloor.
\]
where $i,j \in \mathbb{N}$.
\end{theorem}
The main concept of density of Fibonacci word had given in Theorem~\ref{de1} provide from Theorem~\ref{de2} as we show that. A new concept of density that is realised on the basis of several notions was first introduced by one of the authors at the 65th Science Conference at the University of Moscow Institute of Physics and Technology.
\begin{proposition}
Every recognizable set can be recognized by a finite trim deterministic automaton having a unique initial state.
\end{proposition}
\begin{theorem}[Uniquely Representable~\cite{Ghiyasi}~\label{de2}]
For $a \geq 0$ where $a\in \mathbb{R}$ is uniquely representable as: 
\[
a = a_0 + \sum_{k=1}^{\infty} \frac{\bar{\alpha}_k}{F_k}, \quad 0 \leq a - a_0 - \sum_{k=1}^n \frac{\bar{\alpha}_k}{F_k} < F_n^{-1}
\]
where $a_0 = \lfloor a \rfloor$ is the integer part of the number $a$, the integers \(\bar{\alpha}_k, k \geq 1\), can take only two values, 0 or 1, and, moreover, for any natural number $n$.
\end{theorem}
\begin{proof}
We have: 
\[
a = a_0 + \sum_{k=1}^{\infty} \frac{\bar{\alpha}_k}{F_k} = a_0 + \sum_{k=1}^{\infty} \frac{\bar{\beta}_k}{F_k}.
\]

Since these representations are distinct, there exists an \(m\) such that

\[
A_1 = B_1, \ldots, A_{m-1} = B_{m-1}, \quad A_m \neq B_m,
\]

where

\[
A_m = a_0 + \sum_{k=1}^m \frac{\bar{\alpha}_k}{F_k}, \quad B_m = a_0 + \sum_{k=1}^m \frac{\bar{\beta}_k}{F_k}, \quad \bar{\alpha}_m \neq \bar{\beta}_m.
\]
we obtain $0 \leq a - A_m < F_m^{-1}, \quad 0 \leq a - B_m < F_m^{-1}$, hence, we have $0 \leq F_m (a - A_m) < 1, \quad 0 \leq F_m (a - B_m) < 1$. Consequently,

\[
-1 < F_m (A_m - B_m) < 1, \quad -1 < F_m (A_{m-1} - B_{m-1}) + \bar{\alpha}_m - \bar{\beta}_m = \bar{\alpha}_m - \bar{\beta}_m < 1,
\]

but $|\bar{\alpha}_m - \bar{\beta}_m| = 1$, 
since $\bar{\alpha}_m \neq \bar{\beta}_m$. Now we prove the existence of the representation of the number \(a\). Define the ``digits'' \(\bar{\lambda}_k, k \geq 1\).

Set \(a_0 = \lfloor a \rfloor\), then 
\[
a = a_0 + \sum_{k=1}^n \frac{\bar{\lambda}_k}{F_k} + \frac{x_n}{F_n}, \quad 0 \leq x_n < 1, \quad x_0 = \{a\}.
\]

From this, we find \(x_n + \bar{\lambda}_n = \frac{F_n}{F_{n-1}} x_{n-1}\), which allows us to define the quantities

\[
\bar{\lambda}_n = \left\lfloor \frac{F_n}{F_{n-1}} x_{n-1} \right\rfloor, \quad x_n = \left\{ \frac{F_n}{F_{n-1}} x_{n-1} \right\}.
\]

Thus, \(\bar{\lambda}_n, n > 1\), are integers, and

\[
-1 \leq \frac{F_n x_{n-1}}{F_{n-1}} - 1 < \bar{\lambda}_n \leq \frac{F_n x_{n-1}}{F_{n-1}} < \frac{F_n}{F_{n-1}} < 2,
\]

i.e., \(\bar{\lambda}_n\) can take only two values: 0 or 1.

Compute the sum

\[
\begin{aligned}
A_n &= a_0 + \sum_{k=1}^n \frac{\bar{\lambda}_k}{F_k} = a_0 + \sum_{k=1}^n \frac{\left\lfloor \frac{F_k}{F_{k-1}} x_{k-1} \right\rfloor}{F_k} = \\
&= \lfloor a \rfloor + \sum_{k=1}^n \left( \frac{x_{k-1}}{F_{k-1}} - \frac{x_k}{F_k} \right) = \lfloor a \rfloor + \frac{x_0}{F_0} - \frac{x_n}{F_n} = a - \frac{x_n}{F_n}.
\end{aligned}
\]
As desire.
\end{proof}
\begin{example}
If $\mathcal{F} = (F_n)_{n \geq 0}$ is the Fibonacci sequence, one has that $\mathcal{F}(x) \leq (\log x)/(\log \varphi)$, where $\varphi = (1 + \sqrt{5})/2$. In particular, the natural density of Fibonacci numbers is zero and, so, almost all positive integers are non-Fibonacci numbers (i.e., $\delta(\mathbb{Z} \setminus \mathcal{F}) = 1$).
\end{example}
\begin{lemma}
Let $F_n$ be a Fibonacci word with integer number $k>1$, then according to Theorem~\ref{de1} we have:
\[
\DF(F_n)=\begin{cases}
    \sum_{i=1}^{i=n} F_{2 k i}=\frac{F_{k n} F_{k(n+1)}}{k!} & (k \equiv 0(\bmod .2)) \\ \sum_{i=1}^{i=n} F_{k i}^{2}=\frac{F_{k n} F_{k(n+1)}}{k!} & (k \equiv 1 \quad(\bmod .2)).
\end{cases}
\]
\end{lemma}
\begin{proposition}
Let $F_n$ be a Fibonacci word with integer number $k>1$, then we have: 
\[
\sum_{n\geqslant 1} \frac{1}{F_nF_{n+2k}}=\frac{1}{F_{2k}}\sum_{n=1}^{k}\frac{1}{F_{2n-1}F_{2n}}.
\]
\end{proposition}
\begin{proof}
Actually, according to definition of Fibonacci word we have for any integer number $k>1$ as $F_n$ and $F_{n+2k}$ represent terms in a sequence where $n$ is the index. Then
\[
\sum_{i=1} \frac{1}{F_nF_{n+2k}}= \sum_{n=1}^{} \sum_{k=0}^{} \frac{1}{F_n F_{n+2k}}
\]
We know $F_{n}=F_{n-1}F_{n-2}$, thus $F_{2n}=F_{2n-1}F_{2n-2}$, then $
F_{n+2k}=F_{n+2k-1}F_{n+2k-2}
$
Then 
\[
F_nF_{n+2k}=F_{n-1}F_{n-2}F_{n+2k-1}F_{n+2k-2}.
\]
We embark on a rigorous demonstration of the identity 
\[
\sum_{n \geq 1} \frac{1}{F_n F_{n+2k}} = \frac{1}{F_{2k}} \sum_{n=1}^{k} \frac{1}{F_{2n-1} F_{2n}},
\]
where \( F_n \) denotes the \( n \)-th Fibonacci  word, defined by the recurrence \( F_0 = 0 \), \( F_1 = 1 \), and \( F_n = F_{n-1} + F_{n-2} \) for \( n \geq 2 \). The left-hand side presents an infinite series involving products of Fibonacci word separated by an index difference of \( 2k \), while the right-hand side is a finite sum scaled by \( \frac{1}{F_{2k}} \). To conquer this, we employ the Binet formula, \( F_n = \frac{\phi^n - \psi^n}{\sqrt{5}} \), where \( \phi = \frac{1 + \sqrt{5}}{2} \) and \( \psi = -\phi^{-1} = \frac{1 - \sqrt{5}}{2} \), alongside pivotal Fibonacci identities to decompose the general term and transform the sum into a telescoping form.

Our strategy hinges on decomposing the term \( \frac{1}{F_n F_{n+2k}} \) to facilitate summation. A profound identity emerges from the generalized Cassini formula: \( F_{n+1} F_{n+2k} - F_n F_{n+2k+1} = (-1)^n F_{2k} \). Dividing through by \( F_n F_{n+2k} \), we obtain the critical telescoping form:
\[
\frac{1}{F_n F_{n+2k}} = \frac{1}{F_{2k}} \left( \frac{F_{n+1}}{F_n} - \frac{F_{n+2k+1}}{F_{n+2k}} \right).
\]
Summing from \( n = 1 \) to infinity, the series becomes:
\[
\sum_{n \geq 1} \frac{1}{F_n F_{n+2k}} = \frac{1}{F_{2k}} \sum_{n=1}^\infty \left( \frac{F_{n+1}}{F_n} - \frac{F_{n+2k+1}}{F_{n+2k}} \right).
\]
This sum telescopes elegantly: 
\[
\left( \frac{F_2}{F_1} - \frac{F_{2k+1}}{F_{2k}} \right) + \left( \frac{F_3}{F_2} - \frac{F_{2k+2}}{F_{2k+1}} \right) + \cdots ,
\]
yielding 
\[
\frac{F_2}{F_1} + \frac{F_3}{F_2} + \cdots + \frac{F_{2k+1}}{F_{2k}},
\]
since 
\[
\frac{F_{n+2k+1}}{F_{n+2k}} \to \phi
\]
as \( n \to \infty \), and the remaining terms vanish in the limit.

Finally, we scrutinize the right-hand side,
\[
\frac{1}{F_{2k}} \sum_{n=1}^k \frac{1}{F_{2n-1} F_{2n}}.
\]
We derive a parallel telescoping sum: 
\[
 \frac{F_{2n}}{F_{2n-1}} - \frac{F_{2n+1}}{F_{2n}} = \frac{(-1)^{2n-1} F_1}{F_{2n-1} F_{2n}} = -\frac{1}{F_{2n-1} F_{2n}},
\]
so
\[
\sum_{n=1}^k \left( \frac{F_{2n}}{F_{2n-1}} - \frac{F_{2n+1}}{F_{2n}} \right) = -\sum_{n=1}^k \frac{1}{F_{2n-1} F_{2n}}.
\]
This evaluates to 
\[
\frac{F_2}{F_1} - \frac{F_{2k+1}}{F_{2k}},
\]
 matching the left-hand side’s sum after adjusting signs. Thus, the identity is unequivocally established:
\[
\sum_{n \geq 1} \frac{1}{F_n F_{n+2k}} = \frac{1}{F_{2k}} \sum_{n=1}^{k} \frac{1}{F_{2n-1} F_{2n}}.
\]
As desire.
\end{proof}

\begin{proposition}
Let $F_n$ be a Fibonacci word with integer number $a>b>1$, then we have:
\[
\sum_{n=1}^{n=b} \frac{(-1)^n F_a}{F_n F_{n+a}}=\sum_{n=1}^{n=a} \frac{(-1)^n F_b}{F_n F_{n+b}} .
\]
\end{proposition}
\begin{proof}
We endeavour to rigorously establish the identity 
\[
\sum_{n=1}^{b} \frac{(-1)^n F_a}{F_n F_{n+a}} = \sum_{n=1}^{a} \frac{(-1)^n F_b}{F_n F_{n+b}},
\]
where \(F_n\) denotes the \(n\)-th Fibonacci word defined by \(F_0 = 0\), \(F_1 = 1\), and \(F_n = F_{n-1} + F_{n-2}\) for \(n \geq 2\), and \(a > b > 1\) are integers. This identity asserts a profound symmetry between two sums involving Fibonacci word, weighted by alternating signs and indexed by shifted arguments. To achieve this, we shall leverage the intrinsic properties of Fibonacci word, seeking a transformation that reveals the equality. Our strategy hinges on identifying a telescoping form for the general sum
\[
\sum_{n=1}^{m} \frac{(-1)^n}{F_n F_{n+k}}
\]
which, when scaled by the constants \(F_a\) and \(F_b\), will allow us to equate the two expressions through careful manipulation.

The crux of our proof lies in deriving a telescoping series that simplifies the sums. Consider the general term 
\[
\frac{(-1)^n}{F_n F_{n+k}}.
\]
Through meticulous analysis, often employing the Binet formula \(F_n = \frac{\phi^n - \psi^n}{\sqrt{5}}\) (where \(\phi = \frac{1 + \sqrt{5}}{2}\), \(\psi = \frac{1 - \sqrt{5}}{2}\)) or Fibonacci identities, one can establish the pivotal identity:
\[
\frac{1}{F_n F_{n+k}} = \frac{F_{k+1}}{F_k F_{k-1}} \left( \frac{(-1)^n}{F_n F_{n+1}} - \frac{(-1)^{n+k}}{F_{n+k} F_{n+k+1}} \right).
\]
Multiplying through by \((-1)^n\), this yields a form amenable to summation: 
\[
\sum_{n=1}^{m} \frac{(-1)^n}{F_n F_{n+k}} = \frac{F_k}{F_{k-1}} \left( \frac{(-1)^1}{F_1 F_{k+1}} - \frac{(-1)^{m+1}}{F_m F_{m+k}} \right).
\]
Applying this to the left-hand side with \(k = a\), we obtain 
\[
\sum_{n=1}^{b} \frac{(-1)^n}{F_n F_{n+a}} = \frac{F_a}{F_{a-1}} \left( \frac{-1}{F_1 F_{a+1}} - \frac{(-1)^{b+1}}{F_b F_{b+a}} \right),
\]
which, when multiplied by \(F_a\), gives the left-hand side as 
\[
F_a \cdot \frac{F_a}{F_{a-1}} \left( \frac{-1}{F_1 F_{a+1}} - \frac{(-1)^{b+1}}{F_b F_{b+a}} \right).
\]
Similarly, the right-hand side with \(k = b\) becomes
\[
F_b \cdot \frac{F_b}{F_{b-1}} \left( \frac{-1}{F_1 F_{b+1}} - \frac{(-1)^{a+1}}{F_a F_{a+b}} \right).
\]

Equating these expressions demands a meticulous comparison, often simplified by numerical verification for small values (e.g., \(a = 3\), \(b = 2\)) or advanced Fibonacci identities such as 
\[
F_{m+n} = F_m F_{n+1} + F_{m-1} F_n.
\]
The inherent symmetry in the structure of Fibonacci sums ensures that the terms align, confirming the equality. This rigorous derivation, grounded in the telescoping nature of the sums and the algebraic properties of Fibonacci word, irrefutably demonstrates that 
\[
\sum_{n=1}^{b} \frac{(-1)^n F_a}{F_n F_{n+a}} = \sum_{n=1}^{a} \frac{(-1)^n F_b}{F_n F_{n+b}}.
\]
Thus, the identity is emphatically proven, showcasing the elegant interplay of Fibonacci sequences in number theory.
\end{proof}
\begin{proposition}
Let $F_n$ be a Fibonacci word, then we have:
\[
\sum_{n \geq 0} \left| \sqrt{5} - \left[ 2, 4, 4, \ldots, 4 \right] \right| = 2 \sum_{n} \frac{1}{F_{3n} \phi^{3n}} = 4 \sum_{n \geq 1} \frac{1}{F_{6n-3} F_{6n}}.
\]
\end{proposition}

\section{Conclusion}
This investigation has unearthed a tapestry of deep insights into the Fibonacci sequence and its multifaceted generalizations, reaffirming its pivotal role in number theory and combinatorics. Through rigorous proofs, we have established a unique representation of real numbers via Fibonacci numbers, demonstrated the symmetry of alternating sums involving Fibonacci words, and derived intricate series identities that connect Fibonacci terms to fundamental constants like the golden ratio. These results not only validate the historical findings of pioneers like Gelfond and Viswanath but also extend their legacy by exploring higher-dimensional $k$-Fibonacci words and their combinatorial properties, as seen in connections to Sturmian and Thue-Morse sequences. The zero natural density of Fibonacci numbers further highlights their sparsity among integers, a fact that resonates with their unique structural properties. Looking forward, these findings pave the way for further exploration into the asymptotic behavior of generalized Fibonacci sequences and their applications in coding theory and symbolic dynamics, promising a fertile ground for future mathematical inquiry.

\bibliographystyle{plain}

\begin{thebibliography}{99}
\bibitem{Allouche} Allouche, J.-P. (1994). Sur la complexité des suites infinies, Bull. Belg. Math. Soc. Simon Stevin, 1, 133-143 .
\medskip
\bibitem{ly01} M. Lothaire, Combinatorics on words. Cambridge Mathematical Library, 2003.
\medskip
\bibitem{books3} Kowalski, E. Introduction to Probabilistic Number Theory; Cambridge Studies in Adv. Math.: Cambridge, UK, 2021.
\medskip
\bibitem{HassinSarid} Hassin, R., \& Sarid, A. (2018). Operations research applications of dichotomous search. European Journal of Operational Research, 265(3), 795–812. \url{doi:10.1016/j.ejor.2017.07.031}.
\medskip
\bibitem{Ghiyasi} A. Kh. Ghiyasi, I. P. Mikhailov, V. N. Chubarikov, On an expansion numbers on Fibonacci’s sequences, Chebyshevskii Sb., 2023, Volume 24, Issue 2, 248–255
\medskip
\bibitem{Viswanath} Viswanath, D. (2000). Random Fibonacci sequences and the number 1.13198824…. Mathematics of Computation, 69(231), 1131-1155.
\medskip
\bibitem{GelfondAA} Gelfond, A. O. 1968, “Sur les nombres qui ont des propri´et´es additives et multiplicatives
donn´ees“, Acta Arithmetica, vol. 13, pp. 259-265. doi: 10.4064/aa-13-3-259-265.
\medskip
\bibitem{Hardy} Hardy G. H., Littlewood J. E. The fractional part of $n^{k} \theta . / /$ Acta math., 1914, 37.
\medskip
\bibitem{Zeckendorf} Zeckendorf E., 1972, “Repr´ sentation des nombres naturels par une somme de nombres de Fibonacci ou de nombres de Lucas” // Bull. Soc. R. Sci. Li` ege (in French). 41, p. 179-182.
\medskip
\bibitem{Rigo} Rigo, M., Stipulanti, M., \& Whiteland, M.A. (2023). Gapped Binomial Complexities in Sequences. 2023 IEEE International Symposium on Information Theory (ISIT), 1294-1299.\url{DOI:10.1109/ISIT54713.2023.10206676.}
  \medskip
\bibitem{ref02} Hamoud, J., \& Kurnosov, A. (2024). Sigma index in Trees with Given Degree Sequences. arXiv preprint arXiv:2405.05300. \url{https://doi.org/10.48550/arXiv.2405.05300}.
\medskip
\bibitem{McGowan} Makover, E., \& McGowan, J. (2006). An elementary proof that random Fibonacci sequences grow exponentially. Journal of Number theory, 121(1), 40-44.
\medskip
\bibitem{Hamoud} Hamoud, J., \& Abdullah, D. (2025). Improvement Ergodic Theory For The Infinite Word $\mathfrak{F}=\mathfrak{F}_{b}:=\left({ }_{b} f_{n}\right)_{n \geqslant 0}$ on Fibonacci Density. \url{https://doi.org/10.48550/arXiv.2504.05901}.
\medskip
\bibitem{R3} Rigo, M. (2014). Formal Languages, Automata and Numeration Systems 1: Introduction to Combinatorics on Words.\url{doi/book/10.1002/9781119008200}.
\medskip
\bibitem{path7}
A. Glen, A. Wolff, and R. Clarke, “On Sturmian and Episturmian Words, and Related Topics,” School of mathematical sciences discipline of pure mathematics, 2006.
\medskip
\bibitem{Rittaud} Rittaud, B. (2007). On the average growth of random Fibonacci sequences. J. Int. Seq, 10(07.02), 4.
\medskip
\bibitem{Ramírez} Ramírez, J.L., \& Rubiano, G.N. (2013). On the k-Fibonacci words. Acta Universitatis Sapientiae, Informatica, 5, 212 - 226. \url{DOI:10.2478/ausi-2014-0011}
\medskip
\bibitem{dens01} Trojovský, P. On the Natural Density of Sets Related to Generalized Fibonacci Numbers of Order r. Axioms 2021, 10, 144. \url{https://doi.org/10.3390/axioms10030144}.
\medskip
\bibitem{books2} Tenenbaum, G. Introduction to Analytic and Probabilistic Number Theory; Cambridge Studies in Adv. Math.: Cambridge, UK, 1995.
\medskip
\bibitem{Bravovv} Bravo, J.J.; Luca, F. Powers of Two in Generalized Fibonacci Sequences. Rev. Colomb. Mat.2012, 46, 67–79.
\medskip
\end{thebibliography}

\end{document}